\documentclass[12pt]{amsart}
\usepackage{amsmath,amssymb,amsbsy,amsfonts,amsthm,latexsym,amsopn,amstext,amsxtra,epic, euscript,amscd,indentfirst}
\usepackage{graphicx}
\begin{document}
\renewcommand{\emptyset}{\varnothing}
\newtheorem{theorem}{Theorem}
\newtheorem{conjecture}[theorem]{Conjecture}
\newtheorem{proposition}[theorem]{Proposition}
\newtheorem{question}[theorem]{Question}
\newtheorem{lemma}[theorem]{Lemma}
\newtheorem{cor}[theorem]{Corollary}
\newtheorem{obs}[theorem]{Observation}
\newtheorem{proc}[theorem]{Procedure}
\newtheorem{defn}[theorem]{Definition}
\newcommand{\comments}[1]{} 
%% DEFINITIONS
\def\Z{\mathbb Z}
\def\Za{\mathbb Z^\ast}
\def\Fq{{\mathbb F}_q}
\def\R{\mathbb R}
\def\N{\mathbb N}
\def\cH{\overline{\mathcal H}}
\def\cF{\mathcal F}

\title[A Conjecture on Visible Points]{A Conjectural Inequality for Visible Points in Lattice Parallelograms }

\author{Gabriel Khan}
\address{GK: University of Michigan}
\email{gabekhan@umich.edu}
\author{Mizan R. Khan}
\address{MRK: Department of Mathematics and Computer Science, Eastern Connecticut State University, Willimantic, CT 06226}
\email{khanm@easternct.edu}
\author{Joydip Saha}
\address{JS: Narayanganj, Bangladesh}
\email{joydip.t.saha@gmail.com}
\author{Peng Zhao}
\address{PZ: Department of Mathematics and Computer Science, Indiana State University, Terre Haute, IN 47809}
\email{peng.zhao@indstate.edu}

\date{}

\maketitle

\begin{abstract}
Let $a,n \in \Z^+$, with $a<n$ and $\gcd(a,n)=1$. Let $P_{a,n}$ denote the lattice parallelogram spanned by $(1,0)$ and $(a,n)$, that is,
$$P_{a,n} = \left\{ t_1(1,0)+ t_2(a,n) \, : \, 0\leq t_1,t_2 \leq 1 \right\}, $$
and let 
$$V(a,n) = \# \textrm{ of visible lattice points in the interior of } P_{a,n}.$$ 
In this paper we prove some elementary (and straightforward) results for $V(a,n)$. The most interesting aspects of the paper are in Section 5 where we discuss some numerics and display some graphs of $V(a,n)/n$. (These graphs resemble an integral sign that has been rotated counter-clockwise by $90^\circ$.) The numerics and graphs suggest the conjecture
that for $a\not= 1, n-1$, $V(a,n)/n$ satisfies the inequality 
$$ 0.5 < V(a,n)/n< 0.75.$$
%The basic heuristic for the upper bound is that at least one-fourth of the lattice points inside the parallelogram have both even $x$ and $y$ co-ordinates. However (other than extensive computer calculations) we do not have a heuristic for the lower bound.
\end{abstract}

\section{Introduction and Notation}

In this short paper we describe an ongoing project on visible points in empty lattice parallelograms. We start with some definitions and notation.

A lattice point $(u_1,u_2) \in \Z^2$ is said to be \emph{visible} (from the origin) if $\gcd(u_1,u_2)=1$. Following Apostol~\cite[Theorem 3.9, page 63]{Ap}, for $r >0$ let $N(r)$ denote the number of lattice points in the square $[-r,r]^2$, and let $N^\prime(r)$ denote the number of the lattice points in 
this square that are visible, then
\begin{equation}
 \lim_{r \rightarrow \infty} \frac{N^\prime(r)}{N(r)} = \frac{6}{\pi^2}\,;
\end{equation}
that is, if you pick two integers at random, then there is just over a 60\% chance that they are relatively prime. 

In this paper we count the number of visible lattice points lying inside lattice parallelograms  of a certain form. Specifically, lattice parallelograms that satisfy the following two conditions : (i) other than the four vertices, there are no other lattice points on the boundary of the parallelogram; (ii) one of the vertices is the origin. Such a parallelogram is of the form 
$$ P= \{t_1 \vec{u} + t_2\vec{v} \, : \, 0\leq t_1,t_2 \leq 1 \},$$ 
where 
$\vec{u}, \vec{v} \in \Z^2$ are two visible lattice points that are linearly independent. (Following~\cite{Rez} we call 
a lattice polygon a  \emph{clean} lattice polygon if the only lattice points on its boundary are the vertices.) Our initial goal was to determine the asymptotic estimates that arose in counting the number of visible points lying inside $P$. We 
elaborate on this goal in Section~2.1.

We will use $P_{a,n}$ to denote 
the clean parallelogram with vertices $(0,0),(1,0),$ $(a,n)$ and $(a+1,n)$, where $1\leq a < n$ and $\gcd(a,n) =1$. For a lattice parallelogram, $P$ (with one vertex at the origin) let 
$$V(P)=  \# \textrm{ of visible lattice points in the interior of } P.$$
In the special case when $P=P_{a,n}$, we simply write $V(a,n),$  instead of $V\left(P_{a,n}\right)$.

\section{A Reduction Result}

We begin by showing that questions about visible lattice points inside a clean lattice parallelogram (with one vertex at the origin) can be reduced to studying the visible lattice points inside the appropriate $P_{a,n}$. We also describe the minimum and maximum value of 
$V(P)$. We prove our result via unimodular transformations. These are linear transformations 
$T:\R^n \rightarrow \R^n$ with the key property that they preserve lattice points, that is, $T(\Z^n) = \Z^n$. Furthermore, unimodular transformations preserve visibility.

\begin{defn} A $n\times n$ matrix $M$ is said to be unimodular if $M\in \mathbb{M}_n(\Z)$ with $\det(M) =\pm1$.
A linear transformation $T:\R^n \rightarrow \R^n$ is unimodular if it can be represented by a unimodular matrix.
\end{defn}

\begin{theorem}\label{extreme-case}
Let $P$ be the clean lattice parallelogram spanned by the lattice points $\vec{u},\vec{v}$ with $\textrm{area}(P) =n$. Then there is an unimodular transformation 
$T: \R^2 \rightarrow \R^2$ such that 
$$T(P) = P_{a,n} \textrm{ with }1\leq a <n  \textrm{ and } \gcd(a,n)=1.$$
Furthermore,
$$V(P) = 1 \Longleftrightarrow a= n-1 $$
and
$$V(P) = n-1 \Longleftrightarrow a=1.$$
Geometrically, this means that 
\begin{enumerate}
\item[(i)] All but one of the lattice points in the interior of $P$ are not visible if and only if all of these $n-1$ lattice points lie on the diagonal 
$\vec{u}+\vec{v}.$
\item[(ii)] All the lattice points in the interior of $P$ are visible if and only if all of these lattice points lie on the diagonal 
$$t\vec{u}+(1-t)\vec{v}, \; 0 \le t \le 1.$$
\end{enumerate}
\end{theorem}

\begin{proof}
We begin by observing that there are precisely $(n-1)$ lattice points in the interior of $P$. This follows by combining Pick's theorem with the hypotheses that $\textrm{area}(P)=n$ and that the only lattice points on the boundary of $P$ are the vertices.

%Since $\{\vec{u}\}$ is a primitive set by Corollary~\ref{prim-extension} there is a lattice point $\vec{w}$ such that 
%$\vec{u},\vec{w}$ is a basis of $\Z^2$. 
Without loss of generality we may assume that the pair of lattice points 
$$\vec{u}= (u_1,u_2),\, \vec{v}=(v_1,v_2)$$ 
are positively oriented, that is $\det(\vec{u},\vec{v}) >0$. Since 
$\gcd(u_1,u_2) = 1$, there exist $m_1,m_2 \in \Z$ such that
$$m_1u_1+m_2u_2 =1.$$
We now consider the unimodular matrix 
\begin{equation*}
\left( \begin{array}{cc} m_1 & m_2\\ -u_2 & u_1 \end{array} \right).
\end{equation*}
This unimodular matrix maps $P$ to the parallelogram 
$$ P^{\prime}= \left\{ t_1 (1,0) + t_2(m_1v_1+m_2v_2,n) \; : \, 0\leq t_1,t_2 \leq 1 \right\}.$$
Since $(v_1,v_2)$ is visible, the same holds for $(m_1v_1+m_2v_2,n)$. If $$0 < (m_1v_1+m_2v_2) <n, \textrm{ then }
a=  m_1v_1+m_2v_2 .$$ 
If $ (m_1v_1+m_2v_2) $ does not satisfy this inequality, then we find $k\in \Z$ such that 
$$0 \leq ( m_1v_1+m_2v_2  +kn) <n$$ and act on the parallelogram $P^\prime$ by the unimodular matrix 
\begin{equation*}
\left( \begin{array}{cc}1 & k\\0& 1\end{array} \right)
\end{equation*}
to obtain $P_{a,n}$.

%$$\vec{v} = a\vec{u} + n\vec{w}, \textrm{ with }1\leq a < n, \gcd(a,n)=1.$$ 
%We now invoke Theorem~\ref{unimod-preserve} to infer that there is an unimodular transformation 
%$$L:\Z^2 \rightarrow \Z^2 \textrm{ such that } L(\vec{u}) = (1,0) \textrm{ and } L(\vec{w}) = (0,1).$$
We now consider the special cases $a=1$ and $a=n-1$. All of the lattice points inside $P_{1,n}$ lie on the diagonal 
$y+x=n$. All of these lattice points are visible. 
Turning to $a=n-1$ we note that all of the lattice points inside $P_{n-1,n}$ lie on the diagonal $y=x$. Consequently the only visible point is $(1,1)$.

To complete the proof of the theorem we need to show the following two things. 
For $2 \leq a \leq n-2, P_{a,n}$ contains: (a) at least one non-visible point; (b) at least two visible points.

The proof of item (a) is as follows. Since $2 \le a < n$ and $\gcd(a,n)$=1, there is a positive integer $l$ such that 
$n/2 < la <n$. %Furthermore for $n >4$, $l \le \lceil n/4 \rceil < n/2$. 
The lattice point
$$ \frac{2(n-la)}{n}(1,0) + \frac{2l}{n}(a,n) = (2,2l)$$
is a non-visible lattice point in the interior of $P_{a,n}$.

We turn to the proof of item (b). Since $\gcd(a+1,n) \leq n/2$, there is a lattice point $\vec{s}$ inside $P_{a,n}$ which 
does not lie on the diagonal $y=x$. We now consider the group $G= (\Z \oplus \Z)/(\Z e_1 \oplus \Z (a,n))$.  Let $S$ denote the coset of $G$ that is represented by $\vec{s}$. The lattice point representatives for the cosets  $S$ and $-S$ lie on half-rays that are distinct from each other. Each of these half-rays contains a visible lattice point inside $P_{a,n}$. 
\end{proof}

\subsection{Motivation for studying this problem}

Having proved Theorem~\ref{extreme-case}, we are in a position to expand on the motivation underlying this paper. Unlike the square $[-r,r]^2$, the parallelograms $P_{a,n}$ are long and thin. Thus, it seems plausible to infer that most values of $V(a,n)/n$ should be quite different from $6/\pi^2$.  What we will see is that this inference is false much of the time. That is, for most values of $a$, $V(a,n)/n$ is close to $6/\pi^2$. However, when $a$ is close to 1 or close to $n$ the numerics for $V(a,n)/n$ differ significantly from $6/\pi^2$. The reader may want to jump ahead to Section 5 titled \emph{Some Numerics} to see some graphs of $V(a,n)/n$.

\section{$V(a,n)=V(a^{-1} \mod n,n)$}
 
\begin{proposition}\label{inv-rel}
Let $b= a^{-1} \mod n$. Then the parallelograms $P_{a,n}$ and $P_{b,n}$ are unimodularly equivalent. Consequently, 
$V(a,n) =V(b,n)$.
\end{proposition}

\begin{proof} 
There exists $k\in \Z$ such that $ab+kn=1$. The unimodular transformation $L:\Z^2 \rightarrow \Z^2$, via, 
$$L((1,0)) = (b,n) \textrm{ and }L((0,1)) = (k,-a)$$ 
maps $P_{a,n}$ to $P_{b,n}$.
\end{proof}

The above proposition shows that different values of $a$ that are far apart can take on the same $V$-value. Let us elaborate. Consider the following set in $[0,1]^2$:
$$S= \left\{ \left( \frac{a}{n}, \frac{b}{n} \right) \; : \; a,b,n \in \Z^+, 1\leq a,b, <n, ab \equiv1 \pmod{n} \right\}.$$
A standard argument via Kloosterman sums shows that the points in $S$ are uniformly distributed in the unit square (see~\cite{BK}). A consequence of this is that an element and its multiplicative inverse modulo $n$ can be very far apart. These two elements of $\Z^\ast_n$ take on the same $V$-value.

\section{Counting visible points in $P_{a,n}$}

\begin{theorem}\label{lat-ptP}
The lattice points in the interior of $P_{a,n}$ are 
\begin{equation}\label{eq:lat-pt}
\left\langle\frac{k(n-a)}{n}\right\rangle(1,0) + \frac{k}{n}(a,n)  = \left( \left\lceil\frac{ka}{n}\right\rceil, k \right),
\end{equation} 
with $k=1,\dots,n-1.$
We note that 
$$\left( \left\lceil\frac{ka}{n}\right\rceil, k \right) = \left(\frac{ka}{n}-\left\langle\frac{ka}{n}\right\rangle + 1, k\right).$$
\end{theorem}

From Theorem~\ref{lat-ptP} we get that the lattice points inside the parallelogram $P_{n-a,n}$ are of the form $(\lceil k(n-a)/n \rceil, k)$ for $k=1,\ldots,n-1$.   Consequently  $(\lceil k(n-a)/n \rceil, k)$ is visible if and only if 
$$\gcd(\lceil k(n-a)/n \rceil, k)=1.$$ At this juncture we interject a pretty result that allows us to view the quantity $\gcd(\lceil k(n-a)/n \rceil, k)$ in a couple of different ways.

\begin{proposition} 
Let $a,n$ be positive integers with $a <n$ and $\gcd(a,n)=1$. For $k=1,\ldots,n-1$ ,
\begin{equation} 
  \gcd\left(\left\lceil \frac{k(n-a)}{n}\right\rceil,k\right) =
\gcd\left(\left\lceil \frac{k(n-a)}{n}\right\rceil, \left\lfloor\frac{ka}{n}\right\rfloor\right) =
 \gcd\left(\left\lfloor\frac{ka}{n}\right\rfloor,k\right).
\end{equation}
Furthermore if $n$ is a prime $p$, then 
\begin{equation}
\gcd\left(\left\lceil \frac{k(p-a)}{p}\right\rceil, k\right) 
=\gcd(ka \!\!\!\! \mod p,k).
\end{equation}
\end{proposition}

\begin{proof}
We recall the identity $1-\langle x \rangle=\langle-x\rangle$ when $x\not\in \Z$. On setting $x = k(n-a)/n$ we get
$$ \left\lceil \frac{k(n-a)}{n}\right\rceil+ \left\lfloor\frac{ka}{n}\right\rfloor = \frac{k(n-a)}{n} -\left\langle\frac{k(n-a)}{n}\right\rangle+1 + \frac{ka}{n}-
\left\langle\frac{ka}{n}\right\rangle=k.$$
Consequently, 
$$\gcd\left(\left\lceil \frac{k(n-a)}{n}\right\rceil, \left\lfloor\frac{ka}{n}\right\rfloor\right)= \gcd\left(\left\lceil \frac{k(n-a)}{n}\right\rceil,k\right) = \gcd\left(\left\lfloor\frac{ka}{n}\right\rfloor,k\right).$$

We now consider when $n=p$ is prime. Let $r = ka \mod p$. Then 
$$ ka= p \left\lfloor \frac{ka}{p} \right\rfloor +r,$$
and we infer that 
$$\gcd\left(p \left\lfloor ka/p \right\rfloor,k\right) | \gcd(r,k) \textrm{ and } \gcd(r,k)| \gcd(p \left\lfloor ka/p\right\rfloor,k).$$ Consequently,
$ \gcd(r,k)= \gcd(p \lfloor ka/p \rfloor,k)$. We conclude by observing that since $\gcd(p,k)=1$,  $\gcd(p\lfloor ka/p \rfloor,k) = \gcd(\lfloor ka/p \rfloor,k).$
\end{proof}

We now pose a general question.

\begin{question}
We have proved that $V(1,n)=n-1$ and $V(n-1,n) = 1$. We pose the \emph{vague} question of what can we say about $V(a,n)$ for $a \not =1,n-1$?
\end{question}
 
The next result gives specific examples of $a$ where $V(a,n)/n$ is ``far" away from 0.6.

\begin{theorem} 
Let $n$ be odd. Then we have the following:

\begin{equation}\label{eq:V(2,n)}
V(2,n) = V((n+1)/2,n) = \left\{ \begin{array}{ll}
(3n-3)/4, & n \equiv 1 \pmod{4} \\
(3n-5)/4, & n \equiv 3 \pmod{4}.
\end{array} \right.
\end{equation} 

\begin{equation}\label{eq:V((n-1)/2,n)}
V\left((n-1)/2,n\right)= V(n-2,n)=(n+1)/2.
\end{equation}

%\begin{equation}\label{eq:V((n+1)/2,n)}
%V\left((n+1)/2, n\right) = \left\{ \begin{array}{ll}
%(3n-3)/4, & n \equiv 1 \pmod{4} \\
%(3n-5)/4, & n \equiv 3 \pmod{4}.
%\end{array} \right.
%\end{equation}

%\begin{equation}\label{eq:V(n-2,n)}
%V\left( n-2,n \right)=(n+1)/2.
%\end{equation}

\end{theorem}

\begin{proof}
Since $(n+1)/2 = 2^{-1} \mod n$, we have $V((n+1)/2,n)=V(2,n)$ by Proposition~\ref{inv-rel}. We now invoke 
~\eqref{eq:lat-pt}. The lattice points in the interior of $P_{2,n}$ are 
$$\left\{(\lceil 2k/n\rceil,k) \; : \; k=1,\ldots, n-1 \right\} $$
$$= \left\{ (1,1),(1,2),\ldots, (1,(n-1)/2), (2,(n+1)/2),\ldots, (2,n-1) \right\}.$$
Thus the only non-visible lattice points inside $P_{2,n}$ are of the form $(2,k)$ with $k$ even and $n/2 < k \leq (n-1)$. 
A simple counting argument then gives us~\eqref{eq:V(2,n)}.

Since $n-2= ((n-1)/2)^{-1} \mod n$, we have $V((n-1)/2,n)=V(n-2,n)$ by Proposition~\ref{inv-rel}. By~\eqref{eq:lat-pt} the lattice points in the interior of $P_{(n-1)/2,n}$ are of the form 
\begin{equation}
\left\langle \frac{k(n+1)}{2n} \right \rangle (1,0) + \frac{k}{n}\left(\frac{n-1}{2}\,,n\right) = \left\{ \begin{array}{ll} 
(k/2,k), & k \textrm{ even} \\ ((k+1)/2,k), & k \textrm{ odd,} \end{array} \right. 
\end{equation}
for $k =1,2,\dots,n-1.$ Now, if $k$ is odd, then $(\gcd((k+1)/2,k) =1$; and if $k$ is even, then $\gcd(k/2,k) =1$ if and only if $k=2$. 
By combining these observations we conclude that $V((n-1)/2,n) = (n+1)/2$.
\end{proof}

\noindent We remark that we can extend the above result for any fixed value of $a$. Our next result gives an expression for $V(a,n)/n$ in terms of the Euler phi function and the Mobius function.

\begin{theorem}\label{V(a,n)-for}
Let $\gcd(a,n)=1$. Then 
\begin{equation}\label{eq:phisum} 
\frac{V(a,n)}{n} = \frac{1}{a}\sum_{s=1}^{a}\frac{\varphi(s)}{s} - 
\frac{1}{n}\sum_{s=1}^{a}\sum_{d|s} \left(\left \langle\frac{(s-1)n}{ad}\right\rangle - \left \langle\frac{sn}{ad}\right\rangle\right) \mu(d) - \frac{2}{n},
\end{equation}
where $\varphi$ is the Euler phi function and $\mu$ is the Mobius function.
\end{theorem}

\begin{proof}
We begin by writing $V(a,n)$ has a sum of $a$ terms and then apply an asymptotic formula for each of these terms.
For $s=1,2,3, \ldots, a$, let $l_s$ denote the vertical line (in $\R^2$) $x=s$. We have that 
$$ P_{a,n} \cap \Z^2 = \bigcup_{s=1}^{a} \left(P_{a,n}\cap l_s \cap \Z^2 \right), $$
and consequently
\begin{equation}\label{eq:V-for0}
V(a,n) +2= \sum_{s=1}^a V(s,a,n),
\end{equation}
where 
$$V(s,a,n) =\# \textrm{ of visible lattice points in the intersection } (P_{a,n} \cap l_s).$$
The inclusion of 2 in the LHS of~\eqref{eq:V-for0} arises for the following reason. The term $V(1,a,n)$ counts the lattice point (1,0) and $V(a,a,n)$ 
counts the lattice point $(a,n)$. Both these lattice points are boundary lattice points of $P_{a,n}$. Since $V(a,n)$ only counts the visible points in the interior of $P_{a,n}$, we need to make an adjustment.

From~\eqref{eq:lat-pt} we get 
$$V(s,a,n) = \# \left \{ k \; : \; \frac{(s-1)n}{a} < k \le \frac{sn}{a}, \; \gcd(s,k) =1 \right\}.$$
Let $\varphi(x,N)$ denote the Legendre totient function, that is, the number of positive integers $\le x$ that are relatively prime to $N$. Then
\begin{equation}
V(s,a,n) = \varphi\left(\frac{sn}{a} \, ,s\right) - \varphi \left(\frac{(s-1)n}{a} \, ,s\right),
\end{equation}
and consequently
\begin{equation}\label{eq:Leg-expr}
V(a,n) +2= \sum_{s=1}^a V(s,a,n) = \sum_{s=1}^a \left(\varphi\left(\frac{sn}{a} \, ,s\right) - \varphi \left(\frac{(s-1)n}{a} \, ,s\right)\right).
\end{equation}
A standard identity for the Euler phi function is 
$$ \varphi(m) = \sum_{d|m} \frac{m}{d}\mu(d).$$
(See~\cite[Theorem 2.3]{Ap}.)
It is a simple matter to extend this identity to the following:
$$\varphi(x,N) = \sum_{d|N}\left[ \frac{x}{d} \right] \mu(d)= \sum_{d|N}\left(\frac{x}{d} -\left\langle\frac{x}{d}
\right\rangle \right) \mu(d).$$
Consequently
$$\varphi\left(\frac{sn}{a} \, ,s\right) -\varphi\left(\frac{(s-1)n}{a} \, ,s\right)=
  \frac{n}{a}\sum_{d|s} \frac{\mu(d)}{d}+ \sum_{d|s}\left(\left\langle \frac{(s-1)n}{ad}\right\rangle- \left\langle\frac{sn}{ad}\right\rangle\right)\mu(d).$$ 
Thus,
\begin{equation}\label{eq:V-for1}
V(s,a,n) =\frac{n}{a}\cdot \frac{\varphi(s)}{s} - \sum_{d|s}\left(\left\langle \frac{(s-1)n}{ad}\right\rangle- \left\langle\frac{sn}{ad}\right\rangle\right)\mu(d).
\end{equation}
We obtain~\eqref{eq:phisum} by substituting~\eqref{eq:V-for1} into~\eqref{eq:Leg-expr}.
\end{proof}

\begin{cor}\label{error-V}
If $a=O(n^{1-\epsilon})$ for some $0 < \epsilon < 1$, then
\begin{equation}
\frac{V(a,n)}{n} = \frac{1}{a}\sum_{s=1}^{a}\frac{\varphi(s)}{s}+ o(1).
\end{equation}
\end{cor}

\begin{proof}
The secondary term in~\eqref{eq:phisum} is 
$$\frac{1}{n}\sum_{s=1}^{a}\sum_{d|s} \left(\left \langle\frac{(s-1)n}{ad}\right\rangle - \left \langle\frac{sn}{ad}\right\rangle\right) \mu(d) - \frac{2}{n}.$$
We obtain our asymptotic by showing that the double sum is bounded above by $a\log(a).$ Now,
\begin{equation*}
\left|\sum_{s=1}^{a}\sum_{d|s} \left(\left \langle\frac{(s-1)n}{ad}\right\rangle - \left \langle\frac{sn}{ad}\right\rangle\right) \mu(d) \right| \le 
\sum_{s=1}^{a}\sum_{d|s} | \mu(d) | = \sum_{s=1}^{a}\sum_{d|s, \mu(d) \not= 0}1.
\end{equation*}
Now, 
\begin{equation*}
 \sum_{s=1}^{a}\sum_{d|s,\mu(d) \not= 0}1 =  \sum_{d \le a, \mu(d) \not= 0}\sum_{t \le a/d} 1 = 
\sum_{d \le a, \mu(d) \not= 0}\left\lfloor \frac{a}{d}\right\rfloor = O(a \log(a))= o(n).
\end{equation*}
%We have the inequality 
%$$\left|\sum_{d|s} \left(\left \langle\frac{(s-1)n}{ad}\right\rangle - \left \langle\frac{sn}{ad}\right\rangle\right) \mu(d)\right| \le 2^{\omega(s)},$$
%where $\omega(s)$ equals the number of distinct prime factors of $s$. We will use this inequality to prove the following.
\end{proof}

%The first limit is an immediate consequence of~\eqref{eq:phisum}. The second then follows by invoking the asymptotic
%$$\sum_{n \leq x} \frac{\varphi(n)}{n} = \frac{x}{\zeta(2)} + O(\log(x)).$$
%(See~\cite[exercise 5, page 70]{Ap}.)

We now raise the question of the behavior of $V(a,n)/n$ when $n$ and $a$ are roughly of the same size. We have made very little progress on this. We only have the following simple result.

\begin{proposition}
Let $a,n \in \Z^+$ with $1< a <n$ and $\gcd(a,n)=1$. We make the further assumptions that both $a$ and 
$\lceil n/a \rceil$ are even. Then 
$$V(a,n) \le 0.75 a \cdot \left\lceil \frac{n}{a} \right\rceil.$$
\end{proposition}

\begin{proof}
For $s=1,\ldots a$ the length of the vertical line segment $P_{a,n} \cap l_s$ equals $n/a$. Since $n/a \not\in \Z$ we can infer that for $s=1,\ldots, a,$
$$ \# \left(P_{a,n} \cap l_s \cap \Z^2 \right) = \left\lceil \frac{n}{a} \right\rceil.$$
%and
%$$V(s,a,n) =\frac{n}{a}\cdot \frac{\varphi(s)}{s} - \sum_{d|s}\left(\left\langle \frac{(s-1)n}{ad}\right\rangle- \left\langle\frac{sn}{ad}\right\rangle\right)\mu(d).$$

For $s=1, \ldots, a$, we have the trivial inequalities
$$V(s,a,n) \le \left\lceil \frac{n}{a} \right\rceil, \; s \textrm{ odd}; \; V(s,a,n) \le \frac{1}{2} \left\lceil \frac{n}{a} \right\rceil, \; s \textrm{ even}.$$
For the second inequality we are simply using the fact that the $y$-coordinate of every other lattice point on a vertical line $x=s$ is even. Using these inequalities we obtain 
$$V(a,n) \leq  \sum_{s=1}^a V(s,a,n) \le \frac{a}{2} \cdot \frac{3}{2}\left\lceil \frac{n}{a}\right\rceil = 
0.75a\cdot \left\lceil \frac{n}{a}\right\rceil.$$

%From~\eqref{eq:V-for1} we get the inequality 
%\begin{equation}
%V(s,a,n) \le \frac{n}{a} \frac{\varphi(s)}{s} + 2^{\omega(s)},
%\end{equation}
%where $\omega(s)$ denotes the number of distinct prime factors of $s$.
\end{proof}

\section{Some Numerics}

We have some numerical results for $V(a,n)/n$. For $n\in \Z^{+}$, let 
$$f_n   :  \{a \; : \; a=1, \ldots, n-1, \gcd(a,n)=1 \} \rightarrow [0,1] \textrm{ via } f_n(a) = V(a,n)/n.$$  
The following graphs are the graphs of $f_{499},f_{500},f_{1000}$ and $f_{9973}$. (499 and 9973 are primes.)
\begin{center}
 \includegraphics[height=4cm]{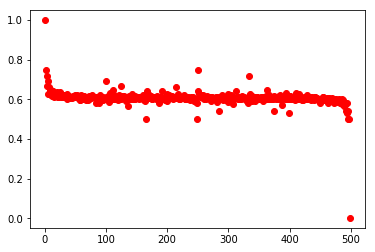}  \includegraphics[height=4cm]{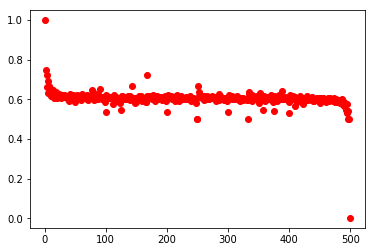}
\end{center}

\begin{center}
The graphs of $f_{499}$ and $f_{500}$.
\end{center}

\begin{center}
 \includegraphics[height=4cm]{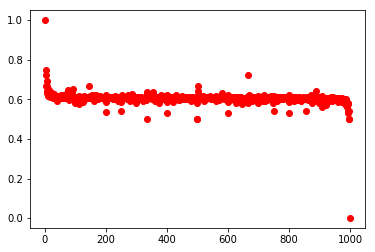}  \includegraphics[height=4cm]{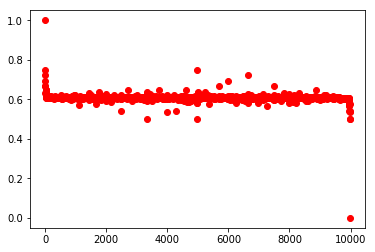}
\end{center}

\begin{center}
The graphs of $f_{1000}$ and $f_{9973}$.
\end{center}

The above result supports the heuristic that one-fourth of the lattice points in the interior of $P_{a,n}$ have both even $x$ and $y$ co-ordinates. Thus the upper bound for $V(a,n)$ should be at most $0.75n$. Unfortunately, we do not have a heuristic for a lower bound. Our computational work suggests the following conjecture.

\begin{conjecture}\index{conjecture!bounds for $V(a,n)/n$}
For $a \not=  1,n-1$,
$$\frac{1}{2} < \frac{V(a,n)}{n}< \frac{3}{4}.$$ 
\end{conjecture}

\section{Further Speculative Remarks} 

The values of $f_n$ oscillate, but only slightly. One can arguably view it as Brownian motion with a negative drift. However,  if one decides to view the graph of $f_n$ from a distance, then it can be described as consisting of 3 parts. Initially it resembles a sharply decreasing function for $a \lessapprox \sqrt{n}$. The $\sqrt{n}$  follows by an application of the well-known \emph{eyeball} metric. It then levels off around $6/\pi^2$ and is ``constant" for most $a$-values. 
Then when $a \gtrapprox n-\sqrt{n}$ it again decreases sharply. (Here we once again invoke the eyeball metric.) There is some ambiguity among the authors about what the eyeball metric tells us. At least, one of us believes that the $\sqrt{n}$ term should be replaced by $\log(n)$. 

In the middle of the graph we see a few outliers where $f_n$ is noticeably bigger or smaller than $6/\pi^2$. At present, we have identified two factors influencing these effects. However, our attempts to make this intuition rigorous have not matured and will be continued in future work.

 Firstly, Proposition~\ref{inv-rel} shows that $$V(a,n) = V(a^{-1} \mod n,n).$$ As such, $a$-values near the endpoints of $[1,n]$ whose modular inverses are in the middle of $[1,n]$ partially explain this outlying phenomena.
 
 Alternatively, we can try to explain this phenomena by trying to get more precise estimates on the terms in Theorem \ref{V(a,n)-for}. The graphs of $f_n$ indicate that we have not obtained sharp estimates for the error analysis in 
Corollary~\ref{error-V}. Using straightforward estimates, we have the asymptotic 
$$\sum_{n \leq x} \frac{\varphi(n)}{n} \approx \frac{x}{\zeta(2)} + O(\log(x)).$$
(See~\cite[exercise 5, page 70]{Ap}.)

With more careful analysis we believe that it is possible to get more refined estimates on the error term. In particular, whenever $a$ is sufficiently large and $a/n$ is not too close to a rational number $p/q$ with $p \ll n$, some preliminary heuristics suggest that the error term 
$$\frac{1}{n}\sum_{s=1}^{a}\sum_{d|s} \left(\left \langle\frac{(s-1)n}{ad}\right\rangle - \left \langle\frac{sn}{ad}\right\rangle\right) \mu(d) \ll 1.$$

More precisely, we assume that the set 
$$ \left \{ \left \langle \frac{kn}{a} \right \rangle ~|~ 1 \leq k \leq Q \right \}$$ 
is sufficiently ``equidistributed" on $[0,1]$ for $1 \ll Q \ll a$. Under this assumption, it is possible to use probabilistic estimates to show that the above expression is $O(n^{-1/2})$. Since our results currently rely on this equidistributional assumption, we postpone a more complete discussion for future work. In any case, we are led to the following (possibly overly optimistic) conjecture.

\begin{conjecture}
Let $0< \rho < 1$ be an irrational number with irrationality measure $2$. 
\[ \lim_{n\to \infty} \frac{V( \lceil{ \rho n \rceil}, n)}{n} = \frac{6}{\pi^2}.\]
\end{conjecture}

In this conjecture, the irrationality measure serves as an attempt to capture the sufficient genericity condition between $a$ and $n$, which quantitatively reflects that $a/n$ is not too close to a sequence of rational numbers with small denominator. Intuitively speaking, this conjecture seeks to formalize the observation that for ``most" pairs of $a$ and $n$, we expect that  
\[ \frac{V(a , n)}{n} \approx \frac{6}{\pi^2}. \]
Another possible way to make this observation rigorous is to consider the function  $f_n(a)$ and try to show that it converges (after renormalization) to $6/\pi^2 $ as $n$ goes to infinity. We know that it does not converge pointwise, but it may converge in $L^p$ or some other norm.

\smallskip

\noindent {\bf Acknowledgements:} We received some interesting feedback from Kevin Ford and Igor Shparlinski. In particular Igor told us how to obtain a stronger version of Corollary~\ref{error-V}.

The first author was partially supported by DARPA/ARO Grant W911NF-16-1-0383 (PI: Jun Zhang).

\end{document}